\documentclass[12pt,reqno]{amsart}
\usepackage[a4paper]{geometry}
\usepackage{tikz}
\usepackage{amssymb,amsmath,amsthm,enumerate,verbatim,bbm}
\usepackage{mathrsfs}

\makeatletter
\@namedef{subjclassname@2020}{\textup{2020} Mathematics Subject Classification} 
\makeatother

\renewcommand{\Re}{\operatorname{Re}}

\newcommand{\abs}[1]{\lvert#1\rvert}
\newcommand{\Abs}[1]{\left\lvert#1\right\rvert}
\newcommand{\norm}[1]{\lVert#1\rVert}

\newcommand{\bbR}{{\mathbb R}}

\newcommand{\bbN}{{\mathbb N}}
\newcommand{\bbZ}{{\mathbb Z}}

\newcommand{\bbD}{{\mathbb D}}

\numberwithin{equation}{section}

\theoremstyle{plain}
\newtheorem{theorem}{\bf Theorem}[section]
\newtheorem*{theorem*}{Theorem 1.1$'$}

\newtheorem{proposition}[theorem]{\bf Proposition}

\theoremstyle{definition}

\theoremstyle{remark}
\newtheorem*{remark*}{\bf Remark}

\begin{document}

\title[Three families of matrices]{Three families of matrices}

\author{Alexander Pushnitski}
\address{Department of Mathematics, King's College London, Strand, London, WC2R~2LS, U.K.}
\email{alexander.pushnitski@kcl.ac.uk}



\date{February 2024}

\begin{abstract}
This paper has an expository nature. 
We compare the spectral properties (such as boundedness and compactness) of three families of semi-infinite matrices and point out similarities between them. The common feature of these families is that they can be understood as matrices of some linear operations on appropriate Hardy spaces. 
\end{abstract}

\maketitle

\section{Introduction}\label{sec.z}
\subsection{Overview}
We denote $n\vee m=\max\{n,m\}$ and let $[n,m]$ be the least common multiple (LCM) of $n$ and $m$. Let $\rho,\tau$ be real parameters and $0<q<1$. 
We consider three families of semi-infinite matrices: 
\begin{align}
A(\tau,\rho)&=
\left\{\frac{q^{\tau(n\vee m)}}{q^{\tau(n+m)/2}}\, q^{\rho(n+m)/2}\right\}_{n,m=0}^\infty,
\label{0a}
\\
B(\tau,\rho)&=\left\{\frac{(nm)^{\tau/2}}{(n\vee m)^\tau}\frac1{(nm)^{\rho/2}}\right\}_{n,m=1}^\infty,
\label{0b}
\\
C(\tau,\rho)&=\left\{\frac{(nm)^{\tau/2}}{[n,m]^\tau}\frac1{(nm)^{\rho/2}}\right\}_{n,m=1}^\infty.
\label{0c}
\end{align}
Our main hero is $C$, while $A$ and $B$ are supporting actors brought in to give some context. The matrix $A$ is the simplest of the three, and should be considered as a warm-up before coming to $B$ and $C$ (although $A$ will make a surprising comeback at the end of the paper). 
Due to the identity
\[
2(n\vee m)-(n+m)=\abs{n-m},
\]
the matrix $A$ can be rewritten in a more concise way as 
\begin{equation}
A(\tau,\rho)=\{q^{\tau\abs{n-m}/2}q^{\rho(n+m)/2}\}_{n,m=0}^\infty;
\label{z1}
\end{equation}
we have chosen to write it as a function of $n\vee m$ in \eqref{0a} in order to emphasise the analogy with $B$ and $C$. 

The matrices $A$, $B$, $C$ satisfy the following scaling relations for any $k\in\bbN$:
\begin{equation}
A_{n+k,m+k}=q^{-\rho k}A_{n,m}, 
\quad
B_{kn,km}=k^{-\rho}B_{n,m},
\quad
C_{kn,km}=k^{-\rho}C_{n,m}.
\label{00}
\end{equation}
Thus, $B$ and $C$ are \emph{homogeneous} in $n,m$ of degree $-\rho$. 
In \eqref{0b} and \eqref{0c}, the homogeneous term is factored out, and a similar factorisation is used in \eqref{0a}.

The matrix $A(\tau,\rho)$ is considered as a (potentially unbounded) linear operator on the Hilbert space $\ell^2(\bbN_0)$, $\bbN_0=\{0,1,2,\dots\}$, while $B(\tau,\rho)$ and $C(\tau,\rho)$ are in the same way considered as linear operators on $\ell^2(\bbN)$, $\bbN=\{1,2,3,\dots\}$. 

We identify the sets of parameters $\rho$ and $\tau$ for which these operators are bounded, compact, positive semi-definite, etc. It turns out that the answers are similar in the three cases. These answers are summarised in a diagram in Figure~1.

One common feature of these three families is that they can be understood as matrices of certain linear operators acting on suitable Hardy spaces of analytic functions. These operators involve shifts and multiplications by reproducing kernels. 

\begin{figure}
\begin{tikzpicture}[scale=0.8]
\draw[fill=gray!30,draw=none]    (0,0) -- (0,2.5) -- (2.5,2.5) -- (2.5,-2.5) -- (0,0);
\draw [thick, ->] (0,-2.5) -- (0,2.5);
\draw [thick, ->] (-2.5,0) -- (2.5,0);
\draw [thick] (0,0) -- (2.5,-2.5);
\node at (2.7,0) {$\rho$};
\node at (0.3,2.8) {$\tau$};
\node at (-1.5,1.5) {$A(\tau,\rho)$};
\end{tikzpicture}
\begin{tikzpicture}[scale=0.8]
\draw[fill=gray!30,draw=none]    (0.5,0) -- (0.5,2.5) -- (2.5,2.5) -- (2.5,-2) -- (0.5,0);
\draw [thick, ->] (0,-2.5) -- (0,2.5);
\draw [thick, ->] (-2.5,0) -- (2.5,0);
\draw [thick] (0.5,0) -- (0.5,2.5);
\draw [thick] (0.5,0) -- (2.5,-2);
\node at (0.5,-0.4) {$1$};
\node at (3,0) {$\rho$};
\node at (0.3,2.8) {$\tau$};
\node at (-1.5,1.5) {$B(\tau,\rho)$};
\end{tikzpicture}
\begin{tikzpicture}[scale=0.8]
\draw[fill=gray!30,draw=none]    (0,0.5) -- (0,2.5) -- (2.5,2.5) -- (2.5,-2) -- (0,0.5);
\draw [thick, ->] (0,-2.5) -- (0,2.5);
\draw [thick, ->] (-2.5,0) -- (2.5,0);
\draw [thick] (0,0.5) -- (2.5,-2);
\node at (0.5,-0.4) {$1$};
\node at (-0.3,0.5) {$1$};
\node at (3,0) {$\rho$};
\node at (0.3,2.8) {$\tau$};
\node at (-1.5,1.5) {$C(\tau,\rho)$};
\end{tikzpicture}
\caption{In the (open) shaded region the corresponding matrix is bounded and compact. 
In the (open) complementary region, the corresponding matrix is unbounded.}
\end{figure}
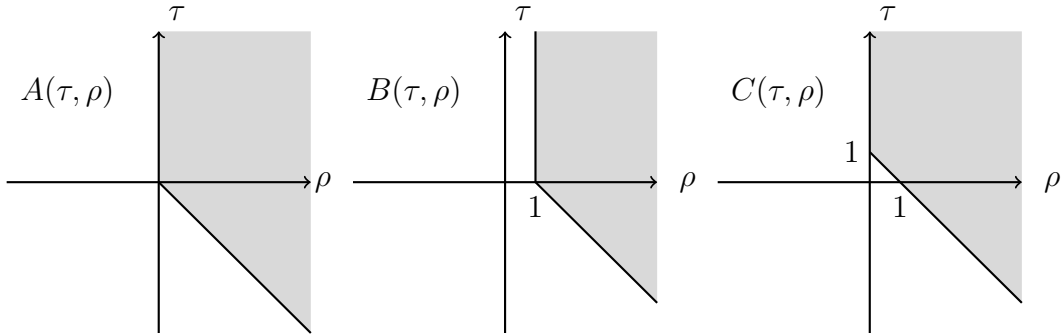

\subsection{The origins of $A$, $B$ and $C$}
A semi-infinite matrix $\mathcal L$ was termed \emph{$L$-matrix} in \cite{BM}, if its matrix elements are of the form
\[
{\mathcal L}_{n,m}=a_{m\vee n}, \quad m,n\in\bbN_0,
\]
where $a=\{a_n\}_{n=0}^\infty$ is a sequence of complex numbers. The term \emph{$L$-matrix} was chosen because $\mathcal L$ has constant elements on $L$-shaped ``corners'' $n\vee m=\text{const}$. 
One of the early appearances of these matrices in the literature can be dated back to  Choi's paper \cite{Choi}, where the $L$-matrix corresponding to $a_n=1/(n+1)$ was called \emph{the loyal companion} of Hilbert's matrix. 
Recently, $L$-matrices and their modifications started attracting some attention, beginning from the paper \cite{BM} by Bouthat and Mashreghi; see \cite{FS,FS2,Brevig,BPP}. One of the key reasons for this was identified by \v{S}tampach, who observed in \cite{FS} that any \emph{non-singular} (i.e. $a_{n+1}\not=a_n$ for all $n$) $\mathcal L$-matrix is the inverse of a Jacobi matrix of a special type. In this connection, one should also mention a closely related and very thorough study \cite{AJPR} of boundedness, compactness and Schatten class conditions of integral operators on the positive semi-axis with kernels $k(x,y)$ depending only on $x\vee y$.

Returning to our matrices $A$, $B$, $C$, we note that  $A(\tau,\rho)$ can be represented as 
\[
A(\tau,\rho)=\left\{q^{(\rho-\tau)n/2}{q^{\tau(n\vee m)}}q^{(\rho-\tau)m/2}\right\}_{n,m=0}^\infty,
\]
which is the product of an $L$-matrix and two diagonal matrices with $q^{(\rho-\tau)n/2}$ on the diagonal; in other words, $A$ is a \emph{weighted} $L$-matrix with an exponential weight. 
Particular cases of $A(\tau,\rho)$ have appeared in different connections in \cite{Sarason}, in \cite[Example 14]{FS} and in \cite{HP,HP2}.
Similarly, $B(\tau,\rho)$ is a weighted $L$-matrix with the power weight $n^{(\tau-\rho)/2}$, and this structure, together with the homogeneity \eqref{00}, makes it quite special. 
Particular cases of $B(\tau,\rho)$ have been studied in the already mentioned papers \cite{BM,FS,Brevig,BPP}. The homogeneity of $B$ tempts one to compare its properties with the properties of the integral operator with the kernel function defined by the the same formula, see the discussion in \cite{BPP,P,Brevig}. 

The origin of $C(\tau,\rho)$ is somewhat different. We first observe that the additive structure of integers plays a crucial role for $A$ and $B$ due to $n\vee m$, while in a similar way the multiplicative structure is essential for $C$ due to $[n,m]$. 
Let us denote by $(n,m)$ the greatest common divisor (GCD) of $n$ and $m$. In 1875, Smith \cite{Smith} computed the determinant of the $N\times N$ matrix 
\[
\{(n,m)\}_{n,m=1}^N
\]
(the answer is the product $\phi(1)\phi(2)\cdots\phi(N)$, where $\phi$ is Euler's totient function). Smith's beautiful paper created a small but distinct research area at the interface of algebra, number theory and spectral theory; see e.g. \cite{BL} and references therein. The matrix $\{(n,m)\}_{n,m=1}^N$ became known as the GCD matrix and $\{[n,m]\}_{n,m=1}^N$ as the LCM matrix, while $\{(n,m)^{-\tau}\}_{n,m=1}^N$ and $\{[n,m]^{-\tau}\}_{n,m=1}^N$ for $\tau\in\bbR$  became known as power GCD and power LCM matrices. Furthermore, one can also consider the weighted power GCD and power LCM versions
\[
\left\{\frac{n^\sigma m^\sigma}{[n,m]^\tau}\right\}_{n,m=1}^N
\quad\text{ and }\quad
\left\{\frac{n^\sigma m^\sigma}{(n,m)^\tau}\right\}_{n,m=1}^N
\]
for $\sigma\in\bbR$. 
Because of the relation
\[
(n,m)[n,m]=nm,
\]
the last two forms are equivalent, up to relabelling the parameters $\tau$ and $\sigma$:
\[
\frac{(nm)^\sigma}{[n,m]^\tau}
=
\frac{(nm)^{\sigma-\tau}}{(n,m)^{-\tau}}.
\]
Separating the homogenous part and relabelling $\rho=\tau-2\sigma$ in the last expression, we arrive at the entries \eqref{0c} of $C(\tau,\rho)$.
While early work on GCD/LCM and related matrices was mostly algebraic, in the last decades some interest arose in their properties as $N\to\infty$, see e.g. \cite{LS,HL,HEL}. 

Another motivation to study the semi-infinite version $C(\tau,\rho)$ arises from the special case $\rho=0$. In this case, $C(\tau,0)$ is related to multipliers acting on the Hardy space of Dirichlet series. This will be explained more fully in Section~\ref{sec.b}. This line of research originates from the paper by Hedenmalm, Lindqvist and Seip \cite{HLS} (although it is implicit in the work of Toeplitz \cite{Toeplitz}) and is partly motivated by connections with analytic number theory. More recently, the matrix $C(\tau,\rho)$ was studied in increasing generality in  \cite{H,H2,HP,HP2}.

\subsection{The structure of the paper}
We consider matrices $A$, $B$ and $C$ in Sections~\ref{sec.a}, \ref{sec.c} and \ref{sec.b} respectively. 
In each section, we start by placing each matrix in the context of the appropriate Hardy space, then we list briefly its spectral properties  depending on $\tau$ and $\rho$ and finally give some brief arguments to explain our claims or point to relevant literature for the proofs. 
Finally, in Section~\ref{sec.d4}, we will explain that $C$ can be considered as an \emph{infinite tensor product} of matrices $A$ over primes.

\section{The matrix $A(\tau,\rho)$}\label{sec.a}

\subsection{The quadratic form of $A(\tau,\rho)$}
Recall that the Hardy space $H^2$ on the unit circle can be identified with the set of functions representable by power series $f(z)=\sum_{n=0}^\infty f_n z^n$ with the finite norm that can be expressed in two alternative ways:
$$
\norm{f}_{H^2}^2=
\int_{-\pi}^\pi \abs{f(e^{i\theta})}^2\frac{d\theta}{2\pi}
=
\sum_{n=0}^\infty \abs{f_n}^2.
$$
The standard (orthonormal) basis in $H^2$ is $\{z^n\}_{n=0}^\infty$. 

\begin{proposition}
For any $\rho\in\bbR$ and $\tau>0$, and for any trigonometric polynomial $f(z)=\sum_{n=0}^\infty f_n z^n$, we have 
\begin{equation}
(1-q^{\tau})\int_{-\pi}^\pi \Abs{(1-q^{\tau/2} e^{i\theta})^{-1}f(q^{\rho/2} e^{i\theta})}^2\frac{d\theta}{2\pi}
=
\sum_{n,m=1}^\infty [A(\tau,\rho)]_{n,m}f_n\overline{f_m}.
\label{a1}
\end{equation}
\end{proposition}
\begin{proof}
Let us substitute the \emph{finite} sum ($f$ is a polynomial!) 
\[
f(q^{\rho/2}e^{i\theta})=\sum_{n=0}^\infty f_n q^{\rho n/2}e^{in\theta}
\]
into the integral in \eqref{a1} and expand. Then we see that it suffices to prove the identity
\[
(1-q^\tau)q^{\rho(n+m)/2}\int_{-\pi}^\pi \Abs{(1-q^{\tau/2} e^{i\theta})^{-1}}^2e^{i(n-m)\theta}\frac{d\theta}{2\pi}=[A(\tau,\rho)]_{n,m}
\]
for every $n$ and $m$. Recalling the expression \eqref{z1} for the matrix entries of $A(\tau,\rho)$, we see that this reduces to 
\begin{equation}
(1-q^\tau) \int_{-\pi}^\pi \Abs{(1-q^{\tau/2} e^{i\theta})^{-1}}^2e^{in\theta}\frac{d\theta}{2\pi}
=
q^{\tau\abs{n}/2}
\label{z4}
\end{equation}
for all $n\in\bbZ$. Now we observe that the expression 
\[
(1-q^\tau)  \Abs{(1-q^{\tau/2} e^{i\theta})^{-1}}^2
\]
is simply the Poisson kernel of the unit disk at the point $q^{\tau/2}$, and so \eqref{z4} represents the well-known expression for the $n$'th Fourier coefficient of the Poisson kernel. However, for expositional purposes we give the proof of \eqref{z4}, because our treatment of matrices $B$ and $C$ below will follow the same pattern.

Since by assumption $\tau>0$, we may expand $(1-q^{\tau/2} e^{i\theta})^{-1}$ in power series and interchange the summation and integration. Thus, for the left hand side of \eqref{z4} we obtain the expression
\[
(1-q^\tau) \sum_{k=0}^\infty \sum_{\ell=0}^\infty q^{\tau(k+\ell)/2} \int_{-\pi}^\pi  e^{i(k-\ell)\theta}e^{in\theta}\frac{d\theta}{2\pi}
=
(1-q^\tau) \sum_{\ell-k=n}q^{\tau(k+\ell)/2}, 
\]
where the sum is over all $k\geq0$ and $\ell\geq0$ with $\ell-k=n$. 
Considering separately the cases $n\geq0$ and $n\leq0$, we easily compute the last sum:
\[
(1-q^\tau) \sum_{\ell-k=n}q^{\tau(k+\ell)/2}
=
(1-q^\tau) \sum_{k=0}^\infty q^{\tau(2k+\abs{n})/2}
=
q^{\tau\abs{n}/2} (1-q^\tau)\sum_{k=0}^\infty q^{\tau k}
=q^{\tau\abs{n}/2},
\]
as required.
\end{proof}
This proposition shows that the matrix $A(\tau,\rho)$ can be identified with the matrix of the linear operator $(1-q^{\tau})T^*T$ in $H^2$ with respect to the standard basis, where $T:H^2\to H^2$ is a combination of a shift and a multiplication:
$$
T: f\mapsto (1-q^{\tau/2} z)^{-1}f(q^{\rho/2} z). 
$$
Observe that $(1-q^{\tau/2} z)^{-1}$ is the reproducing kernel at $H^2$ at the point $q^{\tau/2}$ of the unit disk. 

Of course, this identification makes sense only when $\tau>0$ and $\rho>0$. We shall nevertheless consider the matrix $A(\tau,\rho)$ for all real values of $\tau,\rho$ and ask for what values it corresponds to a bounded operator on $\ell^2(\bbN_0)$.

\subsection{The spectral properties of $A(\tau,\rho)$}\label{sec.b2}

Here we briefly list the spectral properties of $A(\tau,\rho)$; some explanation is given in the next subsection. 

\begin{enumerate}[(1)]
\item
Case $\tau>0$: $A(\tau,\rho)$ is positive semi-definite. 
\begin{enumerate}[(a)]
\item
$\rho>0$: $A(\tau,\rho)$ is bounded, compact and trace class. 

\item
$\rho=0$: $A(\tau,\rho)$ is bounded, but not compact. 

\item
$\rho<0$: $A(\tau,\rho)$ is unbounded. 
\end{enumerate}
\item
Case $\tau=0$: $A(\tau,\rho)$ is formally rank one. It is bounded if and only if $\rho>0$. 
\item
Case $\tau<0$: $A(\tau,\rho)$ is bounded if and only if $\rho+\tau>0$. In this case it is also compact and trace class. 
\end{enumerate}

\subsection{Justification and discussion}
\begin{enumerate}
\item
Let us check that $A(\tau,\rho)$ is positive semi-definite for $\tau>0$.
It suffices to check this for $\rho=0$, since the factor $q^{\rho(n+m)/2}$ does not affect positive semi-definiteness. For $\rho=0$,  this follows from the above proposition, as we have identified $A(\tau,0)$ with $(1-q^{\tau})T^*T$. To put it differently, $A(\tau,0)$ is positive semi-definite because the integral in the left-hand side of \eqref{a1} is non-negative. 
\item[(1a)]
Since $A(\tau,\rho)$  is positive semi-definite, it is trace class if and only if its diagonal entries are summable. It is easy to see that for $\rho>0$ and $\tau>0$ they are indeed summable. 

\noindent
In fact, the analysis of \cite{HP} shows that the eigenvalues of $A(\tau,\rho)$ go to zero exponentially fast: $\lambda_n(A(\tau,\rho))=\mathcal O(q^{\rho n})$ as $n\to\infty$. 

\item[(1b)]
In the case $\rho=0$ we see that $A(\tau,0)=\{q^{\tau\abs{n-m}/2}\}$ is the Toeplitz matrix with the symbol 
$$
\sum_{n=-\infty}^\infty q^{\tau\abs{n}/2}z^n, \quad \abs{z}=1
$$
(which is just the Poisson kernel for the disk at the point $q^{\tau/2}$),
and therefore it has a purely a.c. spectrum that coincides with the range of this symbol.
\item[(1c)]
Observe that the elements on the diagonal of $A(\tau,\rho)$ are $q^{\rho n}$. For $\rho<0$, this sequence goes to infinity as $n\to\infty$; it follows that $A(\tau,\rho)$ is unbounded. 
\item[(2)]
The fact that $A(0,\rho)$  is rank one should not be surprising as in this case it coincides with $C^*C$, where $C$ is the evaluation functional $f\mapsto f(q^{\rho/2})$. 
\item[(3)]
Key to understanding the case $\tau<0$ is the elementary formula
$$
q^{-\tau\abs{n-m}/2}+q^{\tau\abs{n-m}/2}=q^{-\tau(n-m)/2}+q^{\tau(n-m)/2}. 
$$
Multiplying by $q^{\rho(n+m)/2}$, we obtain  
\begin{equation}
[A(\tau,\rho)]_{n,m}=-[A(-\tau,\rho)]_{n,m}+\alpha_n\beta_m+\beta_n\alpha_m,
\label{a4}
\end{equation}
where $\alpha_n=q^{(\rho+\tau)n/2}$ and $\beta_n=q^{(\rho-\tau)n/2}$. 
The last two terms in the right hand side give a rank two operator. 

Now suppose $\rho+\tau>0$; then $\alpha,\beta\in\ell^2(\bbN_0)$ and so the rank two operator is bounded. Since by (1a) the operator $A(-\tau,\rho)$ is bounded, compact and trace class, we obtain that $A(\tau,\rho)$ is also bounded, compact and trace class. 

Conversely, suppose $\rho+\tau\leq0$; then it is easy to see directly that $A(\tau,\rho)$ is unbounded. Indeed, we have $[A(\tau,\rho)]_{n,0}=q^{(\tau+\rho)n/2}\notin\ell^2$. Thus, the image by $A(\tau,\rho)$ of the first vector of the standard basis is not in $\ell^2(\bbN_0)$, hence $A(\tau,\rho)$ is unbounded. 

The rank two reduction formula \eqref{a4} gives a little more. Recall that for $\tau>0$ the operator $A(\tau,\rho)$ is positive semi-definite, and so all of its eigenvalues are positive. Applying simple variational considerations to \eqref{a4}, one can show (see \cite[Section~3.2]{HP2}) that for $\tau<0$, $\rho+\tau>0$, at most one eigenvalue of $A(\tau,\rho)$ is positive, while there are infinitely many negative eigenvalues. 
\end{enumerate}

\section{The matrix $B(\tau,\rho)$}\label{sec.c}

\subsection{The quadratic form of $B(\tau,\rho)$}

\begin{proposition}
For any $\rho\in\bbR$ and $\tau>0$ and for any Dirichlet polynomial $f(s)=\sum_{n=1}^\infty f_n n^{-s}$, we have 
\begin{equation}
\frac{\tau}{2\pi}\int_{-\infty}^\infty\Abs{f(\tfrac{\rho}{2}+it)(it-\tau/2)^{-1}}^2dt
=
\sum_{n,m=1}^\infty f_n\overline{f_m}[B(\tau,\rho)]_{n,m}.
\label{b1}
\end{equation}
\end{proposition}
\begin{proof}
Let us substitute the finite sum
\[
f(\tfrac{\rho}{2}+it)=\sum_{n=1}^\infty f_n n^{-\rho/2}n^{-it}
\]
into the integral in the left hand side of \eqref{b1} and expand. 
Then we see that it suffices to prove the identity
\begin{align*}
(nm)^{-\rho/2}\frac{\tau}{2\pi}\int_{-\infty}^\infty \frac{(n/m)^{-it}}{{\abs{t-i\tau/2}}^2}dt
=
[B(\tau,\rho)]_{n,m}
\end{align*}
for all $n$ and $m$. Observe that here
\[
\frac{\tau}{2\pi}\frac{1}{{\abs{t-i\tau/2}}^2}
\]
is the Poisson kernel of the upper half-plane at the point $i\tau/2$. 
Computing the integral  by the residue calculus yields
\begin{align*}
\frac{\tau}{2\pi}\int_{-\infty}^\infty \frac{(n/m)^{-it}}{(t-i\tau/2)(t+i\tau/2)}dt
&=
e^{-\abs{\log(n/m)}\tau/2}
\\
&=
\frac1{(n/m)^{\tau/2}\vee (m/n)^{\tau/2}}
=
\frac{(nm)^{\tau/2}}{(n\vee m)^\tau}.
\end{align*}
Putting this together, we obtain \eqref{b1}.
\end{proof}

We recall (see e.g. \cite{HLS}) that the Hardy space $\mathscr H^2$ is the Hilbert space of all Dirichlet series $\sum_{n=1}^\infty f_n n^{-s}$ with the finite norm $(\sum_{n=1}^\infty \abs{f_n}^2)^{1/2}$. For Dirichlet polynomials $f$, this norm can be alternatively expressed as
$$
\norm{f}_{\mathscr H^2}^2=\lim_{T\to\infty}\frac1{2T}\int_{-T}^T\abs{f(it)}^2dt.
$$
By Cauchy-Schwarz, all Dirichlet series in $\mathscr H^2$ are analytic in the half-plane $\Re s>1/2$. In fact, the functional of point evaluation at $s$ is bounded on $\mathscr H^2$ if and only if $\Re s>1/2$.

The above proposition tells us that we can identify $B(\tau,\rho)$ with the matrix of the linear operator $\tfrac{\tau}{2\pi}T^*T$ in the standard (orthonormal) basis $\{n^{-s}\}_{n=1}^\infty$
of $\mathscr H^2$, where $T:\mathscr H^2\to L^2(\bbR)$ is the map
$$
T:f\mapsto f(\tfrac{\rho}{2}+it)(it-\tau/2)^{-1}, \quad t\in\bbR.
$$
(Observe that $\tfrac1{2\pi}(it-\tau/2)^{-1}$ is the reproducing kernel for the Hardy class in the right half-plane, at the point $\tau/2$.)
This identification only makes sense for $\tau>0$ and $\rho\geq0$, but below we consider $B(\tau,\rho)$ for all real $\tau$ and $\rho$.

\subsection{The spectral properties of $B(\tau,\rho)$}\label{sec.c2}
\begin{enumerate}[(1)]
\item
Case $\tau>0$: $B(\tau,\rho)$ is positive semi-definite. 

\begin{enumerate}[(a)]
\item
$\rho>1$: $B(\tau,\rho)$ is bounded, compact and trace class. 
\item
$\rho=1$:  $B(\tau,\rho)$ is bounded, but non-compact.
\item
$\rho<1$: $B(\tau,\rho)$ is unbounded. 
\end{enumerate}

\item
Case $\tau=0$: $B(\tau,\rho)$ is formally rank one. It is bounded if and only if $\rho>1$.
\item
Case $\tau<0$: $B(\tau,\rho)$ is bounded if and only if $\rho+\tau>1$.
In this case it is also compact and trace class. 
\end{enumerate}

\subsection{Justification and discussion}
\begin{enumerate}
\item
Let us check that $B(\tau,\rho)$ is positive semi-definite for $\tau>0$.
It suffices to check this for $\rho=0$, since the factor $(nm)^{-\rho/2}$ does not affect positive semi-definiteness. For $\rho=0$,  the matrix $B(\tau,0)$ is positive semi-definite because the left-hand side of \eqref{b1} is non-negative. 
\item[(1a)]
Since $B(\tau,\rho)$ is positive semi-definite and the elements on the diagonal are summable for $\rho>1$, we find that in this case $B(\tau,\rho)$ is trace class. 

It is an interesting question to determine the asymptotics of the eigenvalues of $B(\tau,\rho)$. As far as the author is aware, this question is open.
\item[(1b)]
The spectrum of the operator $B(\tau,\rho)$ for $\rho=1$ was studied in \cite{BPP}.
In particular, it was determined that $B(\tau,\rho)$ has absolutely continuous spectrum $[0,4/\tau]$ with multiplicity one. For all sufficiently large $\tau$, it also has eigenvalues above the continuous spectrum. 
\item[(1c)]
For $\rho<0$ it is easy to see that $B(\tau,\rho)$ is unbounded. Indeed, the diagonal elements are $n^{-\rho}$, which tend to infinity as $n\to\infty$.

For $0\leq \rho<1$, in order to show that $B(\tau,\rho)$ is unbounded, we need to be a little more careful. Take $\sigma$ such that $\tfrac12<\sigma<1-\tfrac12\rho$, and consider the element $x=(x_1,x_2,\dots)\in\ell^2(\bbN)$ with $x_n=n^{-\sigma}$. 
Then 
\begin{align*}
\sum_{n,m=1}^\infty &[B(\tau,\rho)]_{n,m}x_nx_m
\geq
\sum_{n=1}^\infty \sum_{m=1}^{n-1}[B(\tau,\rho)]_{n,m}(nm)^{-\sigma}
\\
&=
\sum_{n=1}^\infty n^{-\frac12(\tau+\rho)-\sigma}\sum_{m=1}^{n-1}m^{-\frac12(\rho-\tau)-\sigma}.
\end{align*}
Since $\sigma<1-\frac12\rho<1-\frac12\rho+\frac12\tau$, for the sum over $m$ we have
\[
\sum_{m=1}^{n-1}m^{-\frac12(\rho-\tau)-\sigma}
\geq
C n^{1-\frac12(\rho-\tau)-\sigma}.
\]
Substituting this into the sum over $n$, we find 
\[
\sum_{n=1}^\infty n^{-\frac12(\tau+\rho)-\sigma}\sum_{m=1}^{n-1}m^{-\frac12(\rho-\tau)-\sigma}
\geq
C\sum_{n=1}^\infty  n^{-2\sigma-\rho+1}=\infty,
\]
and therefore $B(\tau,\rho)$ is unbounded.

\item[(2)]
In this case $B(\tau,\rho)$ corresponds to the point evaluation functional in $\mathscr H^2$ at the point $s=\rho/2$. This functional is bounded if and only if $\rho>1$. 

\item[(3)]
Key to this case is a rank two formula similar to \eqref{a4}. Denoting $n\wedge m=\min\{n,m\}$, we observe that 
$$
\frac1{(n\wedge m)^\tau}+\frac1{(n\vee m)^\tau}
=
\frac1{n^\tau}+\frac1{m^\tau}.
$$
Writing 
$$
\frac1{(n\wedge m)^\tau}=\frac{(n\vee m)^\tau}{(nm)^\tau}
$$
and multiplying by $(nm)^{(\tau-\rho)/2}$, we find
$$
B(\tau,\rho)=-B(-\tau,\rho)+\alpha_n\beta_m+\beta_n\alpha_m,
$$
where $\alpha_n=n^{-(\rho+\tau)/2}$ and $\beta_n=n^{-(\rho-\tau)/2}$. 
The last two terms in the right hand side give a rank two operator.

If $\rho+\tau>1$, this rank two operator is bounded. By (1a), the operator $B(-\tau,\rho)$ is bounded, compact and trace class, and so we find that $B(\tau,\rho)$ is also bounded, compact and trace class. It is also not difficult to see that it has one positive and infinitely many negative eigenvalues. 

Conversely, if $\rho+\tau\leq1$, then $[B(\tau,\rho)]_{n,1}=n^{-(\rho+\tau)/2}\notin\ell^2$, and therefore $B(\tau,\rho)$ is unbounded. 
\end{enumerate}

\section{The matrix $C(\tau,\rho)$}\label{sec.b}

\subsection{The quadratic form of $C(\tau,\rho)$}

Below $\zeta$ is the Riemann zeta-function. 

\begin{proposition}
For any $\rho\in\bbR$, $\tau>2$ and for any Dirichlet polynomial $f(s)=\sum_{n=1}^\infty f_n n^{-s}$, we have
\begin{equation}
\lim_{T\to\infty}\frac1{2T}\int_{-T}^T\Abs{\zeta(\tfrac{\tau}{2}+it)f(\tfrac{\rho}{2}+it)}^2dt
=
\zeta(\tau)\sum_{n,m=1}^\infty f_n\overline{f_m}[C(\tau,\rho)]_{n,m}.
\label{c1}
\end{equation}
\end{proposition}
\begin{proof}
Let us substitute the finite sum
\[
f(\tfrac{\rho}{2}+it)=\sum_{n=1}^\infty f_n n^{-\rho/2}n^{-it}
\]
into the integral in the left hand side of \eqref{c1} and expand. Then we see that it suffices to prove the identity 
\[
(nm)^{-\rho/2}\lim_{T\to\infty}\frac1{2T}\int_{-T}^T
\Abs{\zeta(\tfrac{\tau}{2}+it)}^2(n/m)^{-it}dt
=
\zeta(\tau)
[C(\tau,\rho)]_{n,m}
\]
for all $n$ and $m$. Recalling the definition \eqref{0c} of the matrix entries of $C(\tau,\rho)$, we see that this reduces to
\begin{equation}
\lim_{T\to\infty}\frac1{2T}\int_{-T}^T
\Abs{\zeta(\tfrac{\tau}{2}+it)}^2(n/m)^{-it}dt
=
\zeta(\tau)\frac{(nm)^{\tau/2}}{[n,m]^\tau}.
\label{z3}
\end{equation}
Let us check the last identity. 

Since by assumption $\tau>2$, the Dirichlet series expansion for $\zeta(\tfrac{\tau}{2}+it)$ converges absolutely, and therefore for the left hand side of \eqref{z3} we obtain the expression
\[
\sum_{k=1}^\infty \sum_{\ell=1}^\infty (k\ell)^{-\tau/2}
\lim_{T\to\infty}\frac1{2T}\int_{-T}^T (k/\ell)^{it} (n/m)^{-it}dt.
\]
Here the limit exists and equals one if $km=\ell n$ and zero otherwise. Thus, we find 
\[
\lim_{T\to\infty}\frac1{2T}\int_{-T}^T
\Abs{\zeta(\tfrac{\tau}{2}+it)}^2(n/m)^{-it}dt
=
\sum_{km=\ell n}(k\ell)^{-\tau/2},
\]
where the sum is over all pairs $k,\ell$ with $km=\ell n$. We need to compute the last sum. Denote $d=(n,m)$ (the GCD of $n$ and $m$) and let $n=ad$, $m=bd$. Then the pairs $k$, $\ell$ with $kb=\ell a$ can be listed as $k=ja$, $\ell=jb$, with $j\in\bbN$. Thus, our sum becomes
\begin{align*}
\sum_{km=\ell n}(k\ell)^{-\tau/2}
&=
\sum_{j=1}^\infty (j^2 ab)^{-\tau/2}
=
\zeta(\tau)
(ab)^{-\tau/2}
=
\zeta(\tau)
\left(\frac{nm}{d^2}\right)^{-\tau/2}
\\
&=
\zeta(\tau)
\left(\frac{[n,m]^2}{nm}\right)^{-\tau/2}
=
\zeta(\tau)
\frac{(nm)^{\tau/2}}{[n,m]^\tau},
\end{align*}
as required. 
\end{proof}

Through this calculation, for $\rho\geq0$ and $\tau>2$ we can identify $C(\tau,\rho)$ with the matrix (with respect to the standard basis) of the linear operator $(1/\zeta(\tau))T^*T$ on the Hardy space $\mathscr H^2$ of Dirichlet series; here $T:\mathscr H^2\to\mathscr H^2$ is the map
\begin{equation}
T:f\mapsto \zeta(\tfrac{\tau}{2}+it)f(\tfrac{\rho}{2}+it).
\label{c2}
\end{equation}
Observe that $\zeta(\tfrac{\tau}{2}+it)$ is the reproducing kernel of $\mathscr H^2$ for the point $\tfrac{\tau}{2}$. 

Let us discuss the important special case $\rho=0$. Then $T$ is simply the multiplication by $\zeta(\tfrac{\tau}{2}+s)$ in $\mathscr H^2$, or a \emph{multiplier} in $\mathscr H^2$. The study of such multipliers was initiated in \cite{HLS}. In particular, from the characterisation of bounded multipliers in  \cite{HLS} it follows that for $\rho=0$, the operator $T$ is bounded if and only if $\tau>2$. 

Continuing the discussion of the special case $\rho=0$, we observe that $T$ acts on elements of the standard basis $\{n^{-s}\}_{n=1}^\infty$ as follows:
\[
T: n^{-s}\mapsto \sum_{k=1}^\infty k^{-\tau/2}(kn)^{-s},
\]
and therefore the matrix of $T$ in the standard basis is 
\[
T_{m,n}=
\begin{cases} 
(m/n)^{-\tau/2} &\text{ if $n|m$,}
\\
0 &\text{ otherwise.}
\end{cases}
\]
If the $(n,m)$'th element of a semi-infinite matrix depends on $n/m$, such matrix is known as a \emph{multiplicative Toeplitz matrix}. 
Multiplicative Toeplitz matrices made a first appearance in Toeplitz's 1938 paper \cite{Toeplitz}, which has been until recently unknown to the wider community.  More recently, multiplicative Toeplitz matrices have attracted some attention, see e.g. \cite{H2,NP} and references therein.

\subsection{The spectral properties of $C(\tau,\rho)$}\label{sec.d2}

\begin{enumerate}[(1)]
\item
Case $\tau>0$: $C(\tau,\rho)$ is positive semi-definite. 
\begin{enumerate}[(a)]
\item
$\rho>0$, $\rho+\tau>1$: $C(\tau,\rho)$ is bounded and compact. It is trace class if and only if $\rho>1$. 
\item
$\rho=0$, $\tau>2$: $C(\tau,\rho)$ is bounded but not compact. 
\item
$\rho=0$, $\tau\leq2$: $C(\tau,\rho)$ is unbounded. 
\item
$\rho>0$, $\rho+\tau\leq1$: $C(\tau,\rho)$ is unbounded. 
\item
$\rho<0$: $C(\tau,\rho)$ is unbounded. 
\end{enumerate}

\item
Case $\tau=0$: $C(\tau,\rho)$ is formally rank one. It is bounded iff $\rho>1$. 
\item
Case $\tau<0$: the operator is bounded iff $\rho+\tau>1$. 
\end{enumerate}

\subsection{Justification and discussion}
\begin{enumerate}
\item[(1)]
As was the case for the matrices $A$ and $B$, 
the positive semi-definiteness of $C(\tau,\rho)$ for $\tau>0$ follows from the above proposition since the left hand side of \eqref{c1} is non-negative. 

\item[(1a)]
As established in  \cite{HP}, the eigenvalues of $C(\tau,\rho)$ satisfy the asymptotic relation $\lambda_n=\varkappa n^{-\rho}+o(n^{-\rho})$ as $n\to\infty$. In particular, it follows that $C(\tau,\rho)$ is trace class if and only if $\rho>1$. 
\item[(1b,c)]
As discussed above, for $\rho=0$ the operator $T$ of \eqref{c2} is the multiplier on $\mathscr H^2$ with the symbol $\zeta(\tfrac{\tau}2+s)$, and this multiplier is bounded if and only if $\tau>2$. 
\item[(1d)]
We have $[C(\tau,\rho)]_{n,1}=n^{-(\tau+\rho)/2}$, and for $\tau+\rho\leq1$ this sequence is not an element of $\ell^2$, thus $C(\tau,\rho)$ is unbounded. 
\item[(1e)]
The diagonal elements $[C(\tau,\rho)]_{n,n}=n^{-\rho}\to\infty$ as $n\to\infty$, and so $C(\tau,\rho)$ is unbounded. 
\item[(3)]
The case of $C(\tau,\rho)$ for $\tau<0$ is considered in detail in \cite{HP2}.
In particular, it is also proved in \cite{HP2} that for $\rho+\tau>1$ the matrix $C(\tau,\rho)$ is compact, trace class and has infinitely many positive and negative eigenvalues which obey a certain power law asymptotics. 
\end{enumerate}

\subsection{$C$ as a tensor product of $A$'s}
\label{sec.d4}
For a natural number $n$, let us write its factorisation as a product of powers of primes
\[
n=\prod_{p\text{ prime}}p^{k_p},
\]
where $k_2,k_3,k_5,k_7,\dots$ are non-negative integers labeled by prime numbers, and $k_p=0$ except for finitely many primes $p$. Writing similarly $m=\prod_p p^{j_p}$, we recall that 
\[
[n,m]=\prod_{p\text{ prime}}p^{k_p\vee j_p}.
\]
It follows that the entries of $C(\tau,\rho)$ can be written as 
\[
\frac{(nm)^{\tau/2}}{[n,m]^\tau}\frac1{(nm)^{\rho/2}}
=
\prod_{p\text{ prime}} \frac{p^{\tau(k_p+j_p)/2}}{p^{\tau(k_p\vee j_p)}} \frac1{p^{\rho(k_p+j_p)/2}}.
\]
Inspecting the right hand side, we recognise the matrix entries of $A(\tau,\rho)$ with $q=1/p$. Let us depart slightly from our earlier notation and indicate the dependance of $A$ of the parameter $q$ by writing $A(\tau,\rho;q)$. Then the above factorisation shows that $C$ can be viewed, at least formally, as the infinite tensor product 
\begin{equation}
C(\tau,\rho)=\bigotimes_{p\text{ prime}} A(\tau,\rho;1/p)
\label{fin}
\end{equation}
This point of view is central to \cite{HP,HP2}.

\begin{remark*}
It would be remiss not to mention \emph{Bohr's correspondence} at this point. Bohr's correspondence \cite[Section~2.2]{HLS} is an isomorphism between $\mathscr H^2$ and the Hardy space $H^2(\bbD^\infty)$ of the infinite-dimensional polydisk $\bbD^\infty$. At the heart of this correspondence is the idea to use prime numbers as ``labels'' for the coordinates in $\bbD^\infty$. From this point of view, \eqref{fin} suggests that it is natural to relate $C(\tau,\rho)$ to an operator acting on $H^2(\bbD^\infty)$. We will not elaborate on this as we do not wish this note to become too technical.
\end{remark*}

\section*{Acknowledgements}
The author is grateful to Ole Brevig for useful discussions and encouragement and to Kristian Seip for hospitality during the October 2022 visit to NTNU. The author is also grateful to Titus Hilberdink and Franti\v{s}ek \v{S}tampach for related useful discussions, and to the referee for advice on improving the paper.

\end{document}